\documentclass[11pt]{article}

\usepackage{amsmath,amssymb,amsthm}
\usepackage{amsmath}
\usepackage{amssymb}
\usepackage{amsmath}
\usepackage{mathtools}
\usepackage{amssymb}
\usepackage{enumerate}
\usepackage{mathabx}
\RequirePackage{amsthm,amsmath,amsfonts,amssymb}
\usepackage{tikz-cd} 
\usepackage{mathrsfs}
 \usepackage{caption}
\usepackage{multicol}
\usepackage{calligra}
\usepackage{imakeidx}
\usepackage[pdfencoding=auto]{hyperref}
\usepackage{makeidx}
\newcommand{\R}{\ensuremath{\mathbb{R}}}

\theoremstyle{proclaim}
\newtheorem{theorem}{Theorem}[section]
\newtheorem{lemma}[theorem]{Lemma}

\newtheorem{proposition}[theorem]{Proposition}
\theoremstyle{fancyproclaim}

\theoremstyle{statement}
\newtheorem{remark}[theorem]{Remark}
\newtheorem{definition}[theorem]{Definition}

\theoremstyle{fancystatement}

\numberwithin{equation}{section}

\providecommand{\AMS}{$\mathcal{A}$\kern-.1667em%
\lower.25em\hbox{$\mathcal{M}$}\kern-.125em$\mathcal{S}$}
\usepackage{hyperref}

\newcommand{\subjclass}[2][]{%
  \par\noindent\textbf{Mathematics Subject Classification (#1):} #2\par
}

\newcommand{\keywords}[1]{%
  \par\noindent\textbf{Keywords:} #1\par
}

\newenvironment{acknowledgements}
  {\section*{Acknowledgements}}
  {}

\begin{document}

\title{Exact Sequence of 0-Order Pseudodifferential
Operators on a Lie Groupoid}
\author{Mahsa Naraghi}
\date{}

\maketitle

\begin{abstract}
Associated to a Lie groupoid, there are two $C^*$-algebras: the full and the reduced one. The associated order $0$ pseudodifferential calculus gives rise to multiplier algebras of both. We prove that both associated corona algebras are equal to the natural (commutative) principal symbol algebras. 
\end{abstract}



%\section{Introduction}
%\begin{document}% Do not forget this command!
%\issueinfo{vv}{n}{yyyy} 
% \issueinfo{vv}{n}{yyyy}, vv is the volume, n is the number, yyyy is the year
% leave these information to fixed later by the editorial office
%\commby{Editor}% Editor is the name of the editor who accepted the article
%\pagespan{101}{103}
% \pagespan{bbb}{eee} where bbb is the beginning page, eee is the ending page
%\date{Month dd, yyyy}% This is the date of submission of the article; it
% will be fixed by the editorial office
%\revision{Month dd, yyyy}% This is date(s) of revision of the manuscript; it
% will be fixed by the editorial office
%\title[Short title]{Exact sequence of 0-order pseudodifferential operators on a Lie groupoid}% The title of the article
%\dedicatory{Dedicated to ...}% Only if needed
%\author[MAHSA NARAGHI]{Mahsa Naraghi}% For multiple authors please use {\protec \and} sequence
%\address{Mahsa NARAGHI, Université Paris Cité, Sorbonne Université, CNRS, IMJ-PRG, F-75013 Paris, France}
%\email{mhs.nrgh@gmail.com}

\section*{Introduction}
The exact sequence of pseudodifferential operators on a smooth compact manifold $M$ reads:
$$0\to \mathcal{K}\to \Psi^0(M)\overset{\sigma}{\to} C(S^*M)\to 0$$
where $\mathcal{K}$ is the algebra of compact operators on $L^2(M)$, $\Psi^0(M)$ is the closure of the algebra of order $\le 0$ pseudodifferential operators, and $\sigma$ is the principal symbol map with values on the algebra of continuous functions on the half lines of the cotangent bundle $T^*M$ of $M$. This exact sequence is crucial in the index theory. 

In more singular situations, such as manifolds with boundaries or corners, covering spaces, foliations, \ldots, one often introduces a Lie groupoid taking into account this singularity. Alain Connes (\cite{connes2006theorie}) introduced the pseudodifferential calculus of the  holonomy groupoid of a foliation $(M,F)$ and established an exact sequence
$$0\to C^*(M,F)\to \Psi^0(M,F)\overset{\sigma}{\to} C(S^*F)\to 0.$$
This construction was generalised to any Lie groupoid (\cite{Month997, nisWeXuPn99}). In \cite{nisWeXuPn99}, the associated exact sequence of pseudodifferential calculus on a groupoid $G$ is outlined, which reads:
$$0\to C^*(G)\to \Psi^0(G)\overset{\sigma}{\to} C(S^*A_G)\to 0.$$
In fact, for a general Lie groupoid $G$, there are (at least) two natural groupoid $C^*$-algebras: the full groupoid $C^*$-algebra $C^*(G)$ and the reduced one $C_r^*(G)$. Connes used the reduced one for foliations, whereas Nistor-Weinstein-Xu stated it for both in \cite{nisWeXuPn99}.

\medskip The purpose of the present paper is to give a complete proof of this exact sequence for both the reduced and full $C^*$-algebras inspired by the approach outlined in \cite{debord2019lie}.

The steps for this proof are the following. One proves that
\begin{enumerate}
\item pseudodifferential operators of negative order define elements of $C^*(G)$ (prop. \ref{prop2.1});
\item pseudodifferential operators of zero order define multipliers of $C^*(G)$ (prop. \ref{03});

This gives a morphism $C(S^*A_G)\to \Psi^*(G)/C^*(G)$.

\item Extension of the morphism $\lambda :C^*(G)\to C^*_r(G)$ to multipliers yields a morphism $\Psi^*(G)/C^*(G)\to \Psi_r^*(G)/C_r^*(G)$.

\item Finally, we construct for every element  $(x,\xi)\in S^*A_G$ a sequence $(\phi_n)$ of unit vectors of $L^2(G_x)$ such that $\langle \phi_n, \lambda_x(P)(\phi_n)\rangle\to \sigma(P)(x,\xi)$ if $P$ is a pseudodifferential operator of order $0$ and  $\langle \phi_n, \lambda_x(P)(\phi_n)\rangle\to 0$ if $P\in C^*(G)$.

This gives a morphism $\Psi_r^*(G)/C_r^*(G)\to C(S^*A_G)$.

\item Looking at these morphisms on the dense set of pseudodifferential operators of order $0$, shows that they are all isomorphisms.
\end{enumerate}

%*****************************************************************************

\section{Preliminaries}

In the first section, we recall some facts about Lie groupoids, their $C^*$-algebras and we list the facts about the pseudodifferential calculus that will be used in the sequel.

\subsection{Lie groupoids}
We recall here the main definitions and fix some notation on Lie groupoids. There are several books and survey papers in which details can be found -- see \emph{e.g.} \cite{mackenzie2005general}.

\begin{definition}
A \textit{groupoid}\index{groupoid} $G$ is a small category with inverses. The set of objects of the groupoid is denoted by $G^{(0)}$, and the set of morphisms is denoted by $G$. Each morphism $\gamma \in G$ has a \emph{source} object $s(\gamma)\in G^{(0)}$ and a \emph{range} object $r(\gamma)\in G^{(0)}$. To every object $x\in G^{(0)}$ is assigned a \emph{unit} morphism $1_x$, and the map~$u:x\mapsto 1_x$ splits both $r$ and $s$ (\emph{i.e.}  $r(1_x)=s(1_x)=x$).   The set of composable morphisms is $G^{(2)}=\{(\gamma_1,\gamma_2)\in G\times G ; s(\gamma_1)=r(\gamma_2)\}$. The \emph{composition} of morphisms is a map $m:G^{(2)}\to G$, $(\gamma_1,\gamma_2)\mapsto \gamma_1\gamma_2$ such that
\begin{itemize}
\item $1_{r(\gamma)}\gamma=\gamma 1_{s(\gamma)}=\gamma$ for every $\gamma \in G$;
    \item $s(\gamma_1\gamma_2)=s(\gamma_2)$ and $r(\gamma_1\gamma_2)=r(\gamma_1)$ for all $(\gamma_1,\gamma_2)\in G^{(2)}$;
    \item $\gamma_1(\gamma_2\gamma_3)=(\gamma_1\gamma_2)\gamma_3$ for all $\gamma_1,\gamma_2,\gamma_3\in G$ such that $s(\gamma_1)=r(\gamma_2)$ and $s(\gamma_2)=r(\gamma_3)$.
\item Every $\gamma\in G$ has an \emph{inverse} $\gamma^{-1}$ such that $s(\gamma^{-1})=r(\gamma), r(\gamma^{-1})=s(\gamma)$ and $\gamma^{-1}\gamma=1_{s(\gamma)}, \gamma\gamma^{-1}=1_{r(\gamma)}$.
\end{itemize}
\end{definition}

We often identify $G^{(0)}$ with a subset of $G$ by means of the unit map $x\mapsto 1_x$.
We write $G\overset{r}{\underset{s}\rightrightarrows} G^{(0)}$ or just $G\rightrightarrows G^{(0)}$ to mean $G$ is a groupoid with set of objects $G^{(0)}$.
 
\begin{definition}
A \textit{Lie groupoid} is a groupoid $G\rightrightarrows M$ with a smooth manifold structure on $G$ and $M$ such that  $s,r:G\to M$ are smooth submersions, multiplication is a smooth map and the unit map is smooth. 

As $s$ and $r$ are submersions, $G^{(2)}=(s\times r)^{-1}(\Delta(M)$ is an embedded submanifold of $G\times G$ - this gives a meaning to the fact that $m$ is smooth.

The  unit map $u$ is an immersion and $M$ is an embedded submanifold of $G$. 

The inverse map $\gamma \mapsto \gamma^{-1}$ is a diffeomorphism.  
\end{definition}

We denote by $N$ the normal bundle of the inclusion $M\subset G$: for $x\in M$, we have $N_x=T_xG/T_xM$. This bundle has many features which are not essential for us here. In particular, the space of smooth sections of this bundle is the \emph{Lie algebroid} of $G$ (see \cite{mackenzie2005general}). 

\subsection{\texorpdfstring{$C^*$}{C*}-algebras of a Lie Groupoid.} (See \cite{connes2006theorie},  \cite{renault1980groupoid}.) 

\paragraph{The convolution algebra.} Let $G$ be a Lie groupoid. There is a natural convolution on the space of smooth functions with compact support $C^\infty _c(G)$ defined by $f_1\star f_2 (\gamma)=\int_{\gamma_1\gamma_2=\gamma}f_1(\gamma_1)f_2(\gamma_2)$. 

To perform this convolution product, we need a smooth family of measures on the sets $m^{-1}(\{\gamma\})=\{(\gamma_1,\gamma_2)\in G^{(2)};\ \gamma_1\gamma_2=\gamma\}$. Note that the map $(\gamma_1,\gamma_2)\mapsto \gamma_1$ (\emph{resp.}  $(\gamma_1,\gamma_2)\mapsto \gamma_1$) is a diffeomorphism between  $m^{-1}(\{\gamma\})$ and $G^{r(\gamma)}$ (\emph{resp.}  $ G_{s(\gamma)}$). Such a measure is provided by a smooth left Haar system: a smooth family of Lebesgue measures on the manifolds $G^x$ ($x\in M$) invariant under the diffeomorphisms $L_\gamma:\gamma_1\mapsto \gamma \gamma_1$ between $G^{s(\gamma)}$ and $G^{sr(\gamma)}$.

Alain Connes \cite{connes2006theorie} gave a nice canonical meaning of this convolution using sections of a natural bundle of half-densities: we replace functions by smooth compactly supported sections of the half density bundle $\Omega^{1/2}(\ker ds\times \ker dr)$.  See \cite{CrainicMestre2025, NSV} for details of this construction.

\medskip With a natural  smooth choice of measures, or using half densities, the convolution is associative. 

The involution  $f^*(\gamma)=\overline{f(\gamma^{-1})}$ turns  $C_c^\infty (G)$ into a *-algebra.

\paragraph{The full C*-algebra}

\begin{definition}
A representation of a Lie groupoid $G$ on a complex Hilbert space $H$ is a $*$-algebra homomorphism $\Phi: C_c(G)\to B(H)$ which is bounded with respect to the natural inductive limit topology.
\end{definition}

\sloppy It turns out (\emph{cf.} \cite{renault1980groupoid}) that the operator norm of $\Phi(f)$ can be controlled by an $L^1$-type estimate so that $\sup(\|\Phi(f)\|)$ over all the possible continuous $*$-representations $\Phi$ is finite for every $f\in C_c(G)$.

\begin{definition}
The maximal norm of $f\in C_c^\infty(G)$,  denoted as $\|\cdot\|_{\max}$, is the supremum of the operator norm $\|\Phi(f)\|$ over all continuous representations $\Phi$ of $G$. The completion of $C_c^\infty(G)$ with this norm is called the \index{maximal $C^*$-algebra}maximal $C^*$-algebra of $G$ and is denoted with $C^*_{\text{max}}(G)$. In the following, we denote it by $C^*(G)$.
\end{definition}

\medskip\paragraph{The reduced $C^*$-algebra.}
For every $x\in M$ we may represent the $*$-algebra $C_c^\infty (G)$ on the Hilbert space $L^2(G^x)$.

For $f\in C_c(G)$ and $\xi\in C_c^\infty(G^x)$, we may form $f\star \xi\in C_c^\infty(G^x)$ defined by $(f\star \xi)(\gamma)=\int _{m^{-1}(\{\gamma\})}f(\gamma_1)\xi(\gamma_2)$. 

In this way, we obtain a $*$-morphism $\lambda^x:C_c(G)\to \mathcal{L}(L^2(G^x))$.

\begin{definition}
The reduced norm of $f\in C_c(G)$ is defined by $\|f\|_r=\sup \|\lambda^x(f)\|$. The completion of $C_c(G)$ with respect to $\|\ \|_r$ is the reduced $C^*$-algebra of $G$ and is denoted with $C^*_r(G)$.
\end{definition}

\begin{remark}
\begin{itemize}
    \item For every $f\in C_c^\infty(G)$,  $\|f\|_r\leq \|f\|_{\max}$.
    \item The identity map on $C_c^\infty(G)$ induces a surjective $*$-homomorphism from $C^*(G)$ to $C_r^*(G)$.
\end{itemize}
\end{remark}

\subsection{Pseudodiﬀerential operators on Lie groupoids} Associated with every Lie groupoid is a pseudodifferential calculus, generalising the pseudodifferential calculus on a manifold. This calculus was first introduced by Connes for foliations (\cite{connes2006theorie}) and  generalised by several authors (\cite{Month997, nisWeXuPn99}...).

We recall here just a few facts about this calculus which will be useful for us. A detailed account on this calculus can be found in \cite{ NSV, nisWeXuPn99,  vassout2006unbounded}.

\medskip
Let $G\rightrightarrows M$ be a Lie groupoid. Recall we denote by $N=T_MG/TM$ the normal bundle of the submanifold $M$ of $G$, $p :T_MG \to N$ the quotient map, and $N^*$ its dual bundle. 

\begin{itemize}
\item Denote by $\Sigma_c^m{(N^*)}$ the space of homogeneous compactly supported symbols: it is the space of smooth positively $m$-homogeneous functions on $N^*\setminus M$ that vanish outside some $q^{-1}(K)$ where $K$ is a compact subset of $M$ and $q :n^* \to M$ the structural map,  \emph{i.e.} smooth functions $f(x,\xi)$ with $x\in M$, $\xi\in N^*_x,\ \xi\ne 0$, $f(x,\lambda \xi)=\lambda^mf(x,\xi)$ for every $\lambda \in \R_+^*$, and $f(x,\xi)=0$ if $x\not\in K$. 

\item
In the algebra of multipliers of the algebra $C_c^\infty(G)$, there is, for every $m\in \mathbb{Z}$, a space $\mathcal{P}^m_c(G)$ called the space of  ``smooth, compactly supported pseudodifferential operators of order  $m$''.  The space $\mathcal{P}^m_c(G)$ is closed under $P\mapsto P^*$.

\item For every $m\in \mathbb{Z}$, we have $\mathcal{P}^{m-1}_c(G)\subset \mathcal{P}^{m}_c(G)$.
Moreover, there is a (principal) symbol map, which is a linear map $\sigma_m:\mathcal{P}^m_c(G)\to \Sigma_c^m(N^*)$ which is  surjective with kernel $\mathcal{P}^{m-1}_c(G)$.

\item For $m,n\in \mathbb{Z}$, $P\in \mathcal{P}^{m}_c(G)$ and $Q\in \mathcal{P}^{n}_c(G)$, we have $PQ\in \mathcal{P}^{m+n}_c(G)$ and $\sigma_{m+n}(PQ)=\sigma_m(P)\sigma_n(Q)$, and $\sigma_m( P^*)=\overline{\sigma_m( P)}$.

\item
$\bigcap _{m\in \mathbb{Z}}\mathcal{P}^{m}_c(G)=C_c^{\infty}(G)$.

\item For every $a\in C^{\infty}(M)$, define the multiplier $M_a$ of $C_c^\infty(G)$ by setting $(M_af)(\gamma)=a(r(\gamma))f(\gamma)$ and $(fM_a)(\gamma)=f(\gamma)a(s(\gamma))$. If $a\in C_c^\infty(M)$, then $M_a\in \mathcal{P}^{0}_c(G)$, and $\sigma_0(M_a)(x,\xi)=a(x)$ (for every $x\in M$ and $\xi\in N_x\setminus\{0\}$).
\end{itemize}

To construct $\mathcal{P}^{m}_c(G)$ one proceeds as follows:

Let $\theta : U' \to U$ be an exponential map, which is a diffeomorphism from an open neighborhood $U'\subset N$ of the zero section of $N$  to an open neighborhood of $M$ in $G$, such that
\begin{itemize}
    \item $\theta(x,0)=x$ for all $x\in M$ and $r(\theta(x,X))=x$ for $(x,X)\in U'$.
    \item $p_x(\mathrm d\theta_{(x,0)}(X))=X $ for all $x\in M$ and $X\in N_x$, where $p : T_M G \to N$ is the quotient map.
\end{itemize}

\begin{definition}\label{340}
An element $P\in \mathcal{P}^{m}_c(G)$ is a  multiplier $P=P_0+\kappa$ of the $*$-algebra $C_c^\infty(G)$ where $\kappa\in C_c^\infty(G)$ and  multiplication by $P_0$ is given by an expression of the form
\begin{equation*}
  (P_0\star f)(\gamma)
  =\frac{1}{(2\pi)^n}\int_{N^*_{r(\gamma)}}\left(\int_{G_{s(\gamma)}}
    \mathrm e^{-\mathrm i\langle \theta^{-1}(\gamma \eta^{-1}),\xi\rangle}\,
    a\big(r(\gamma),\xi\big)\,\chi(\gamma \eta^{-1})\,f(\eta)\,
    \mathrm d\eta\right)\,\mathrm d\xi,
\end{equation*}
where  $\chi$ is a cut-off function with $\chi(\eta)=0$ for $\eta\not\in U$,  and $\chi(\eta)=1$ for $\eta$ in a neighborhood of $M$ in $U$, $N^*$ denotes the dual bundle of $N$, $x=r(\gamma)$, $n=\dim N_x$ and $a\sim \sum_{k=0}^{+\infty}a_{m-k}$ is a \emph{classical symbol} of order $m$ (with  $a_j(x,\xi)$ positively homogeneous of order $j$ in $\xi$, see e.g. \cite{taylor1981pseudodiff} for more details on classes of symbols). 

One has $\sigma_m(P)=a_m$.
\end{definition}

%*****************************************************************************
\section{Exact Sequences of Pseudodifferential Operators of Order 0}
We now come to the main constructions of the paper.

\subsection{Pseudodifferential operators of negative order}

\begin{proposition}\label{prop2.1}
For $m<0$, we have $\mathcal{P}^{m}_c(G)\subset C^*(G)$. More precisely, there is an algebra morphism $j:\mathcal{P}^{m}_c(G)\to C^*(G)$ whose restriction to $C_c^\infty(G)$ is the inclusion of $C_c(G)$ in the $C^*$-algebra of $G$.
\end{proposition}

\begin{proof}
Denote by $n$ the rank of the bundle $N$. Let $P= P_0+\kappa \in \mathcal{P}^{m}_c(G)$, where $P_0$ is associated to a classical symbol $a$ of order $m$ with $m<-n$. By definition of classical symbols, the map $\xi \to \|\xi\|^{-m} a(x, \xi)$ is bounded on $N_x$, so that the following map $g$ is continuous as the "Fourier transform" of an integrable function. 
$g(\eta)=  \frac{\chi(\eta)}{(2\pi)^n}\int_{N^*_{r(\eta)}G} 
    \mathrm e^{-\mathrm i\langle \theta^{-1}(\eta),\xi\rangle}\,
    a\big(r(\eta),\xi\big)\, )\,\mathrm d\xi$.
    
Hence we have $g\in C_c(G)$ and $(P_0\star f)=g\star f$ , $f\star P_0=f\star g$. We may then put $j(P_0)=j(g)$.
     
In particular, for  $f\in C_c^\infty (G)$, we have $\|P\star f\|_{C^*(G)}\le \|g\|_{C^*(G)}\|f\|_{C^*(G)}$.

Now, if $m<-n/2$, since $j(P^*P)\in C^*(G)$, it follows that, for every $f\in C_c^\infty(G)$, $$\|P\star f\|_{C^*(G)}^2=\|f^*\star P^*\star P\star f\|_{C^*(G)}\le \|f^*\|_{C^*(G)}\|P^*P\|_{C^*(G)}\|f\|_{C^*(G)},$$ and it follows that $P$ extends to a left (and in the same way right) multiplier $j(P)$ of $C^*(G)$. But since $j(P)^*j(P)=j(P^*P)\in C^*(G)$, we have $j(P)\in C^*(G)$.

We find by induction that, if $2^km<-n$, then $j(P)\in C^*(G)$.
\end{proof}

\subsection{Pseudodifferential operators of order $0$ are multipliers}

Note that the algebra $\Sigma_c^0(N^*)$ identifies with the algebra $C_c^\infty(S^*N)$ of smooth functions with compact support on the space $S^*N$ of half lines in $N^*$.

\begin{proposition}\label{03}
Let $G\rightrightarrows M$ be a Lie groupoid and $P\in \mathcal{P}_{0,c}(G)$. Let $\sigma _0(P) \in C_0(S^*N)$ denote the principal symbol of $P$,  
\begin{enumerate}
    \item $P$ is a multiplier of $C^*(G)$;
    \item  $\|P\|_{\mathcal{M}(C^*(G))/C^*(G)}\leq \| \sigma_0(P)\|_\infty$.
\end{enumerate}
\end{proposition}
\begin{proof}
The map $\sigma_0(P)$ is bounded, so up to dividing $P$ by a constant, we can assume $\| \sigma_0(P)\|_\infty\leq 1$. 

\medskip
Let $t\in [0,1)$. Let $a\in C^\infty_c(M)$ be a non-negative function  such that $\|a\|_\infty=1$ and $a(x)=1$ whenever $\sigma_0(P)(x,\xi)\neq 0$.

\medskip
Then the map $\displaystyle b : (x,\xi)\mapsto a(x)\sqrt{1-|t\sigma(P)(x,\xi)|^2}$ is a smooth function with compact support on $S^*N$, hence the principal symbol of a pseudodifferential operator $Q$ of order $0$.  \\ Then, $\sigma_0(t^2P^*P+Q^*Q)=a^2$ so that $M_a^2-(t^2 P^*P+Q^*Q)\in \mathcal{P}^{-1}_c(G)\subset C^*(G)$. Thus $(t^2 P^*P+Q^*Q)$ is a multiplier.

For $f\in C_c^\infty(G)$, we have (in $C^*(G)$) $\|tPf\|^2=t^2\|f^*P^*Pf\|\le \|f^*(t^2 P^*P+Q^*Q)f\|\le \|(t^2 P^*P+Q^*Q)\|\|f\|^2$. 

In the same way $P$ is a right multiplier (or $P^*$ a left multiplier). Hence $P \in \mathcal{M}(C^*(G))$.

This also gives the following equality of norms 
\begin{equation*}
    \|t^2 P^*P+Q^*Q\|_{\mathcal{M}(C^*(G))/C^*(G)}=\|M_a^2\|_{\mathcal{M}(C^*(G))/C^*(G)}=\|a\|^2_\infty=1.
\end{equation*}

%is of the form $a^2+R$ where $R$ is of negative order thus $a^2-(PP^*+QQ^*)\in C^*(G)$.
%It yields $\sigma (t^2PP^*+QQ^*)=a^2$ and $a^2-(t^2PP^*+QQ^*)\in C^*(G)$ since it is of negative order. 

%Note that there is an injective $*$-morphism 
%$C_0(G^{(0)})\rightarrow \mathcal{M}(C^*(G))$ which is of course norm decreasing. 

Thus, by positivity, 
\begin{align*}
 \|tP\|^2_{\mathcal{M}(C^*(G))/C^*(G)}&=\|t^2 P^*P\|_{\mathcal{M}(C^*(G))/C^*(G)}   \\
 &\leq \|t^2 P^*P+Q^*Q\|_{\mathcal{M}(C^*(G))/C^*(G)} \\
 &\le  1.
\end{align*}

This being true for every $t\in [0,1)$, we find $\|P\|_{\mathcal{M}(C^*(G))/C^*(G)} \le 1$.
\end{proof}

We then get the folllowing proposition.
%Let $\Psi^*(G)$ denote the closure of the $\mathcal{P}_{0,c}$ in $\mathcal{M}(C^*(G))$. 
% Since the morphism $\sigma :\mathcal{P}_{0,c}\to C_0(S^*N)$ has dense range, we deduce the existence and uniqueness of a $C^*$-morphism $\alpha :C_0(S^*N)\to \Psi^*(G)/C^*(G)$ making the diagram 

% \begin{center}
% \begin{tikzcd}[row sep=huge, column sep=huge]
%     \mathcal{P}_{0,c} \arrow[rd, "q\circ j"'] \arrow[r, "\sigma"] & C_0(S^*N) \arrow[d, "\alpha"] \\
%     & \Psi^*(G) /C^*(G)
% \end{tikzcd}
% \end{center}
% commute,  where $j$ denotes the canonical injection of $\mathcal{P}_{0,c}$ in $\mathcal{M}(C^*(G))$ and $q: \Psi^*(G)\rightarrow\Psi^*(G)/C^*(G)$ denotes the quotient map.
\begin{proposition}
Let $\Psi^*(G)$ denote the closure of the $\mathcal{P}_{0,c}$ in $\mathcal{M}(C^*(G))$. There exists a unique $C^*$-morphism $\alpha :C_0(S^*N)\to \Psi^*(G)/C^*(G)$ making the diagram 

\begin{center}
\begin{tikzcd}[row sep=huge, column sep=huge]
    \mathcal{P}_{0,c} \arrow[rd, "q\circ j"'] \arrow[r, "\sigma"] & C_0(S^*N) \arrow[d, "\alpha"] \\
    & \Psi^*(G) /C^*(G)
\end{tikzcd}
\end{center}
commute,  where $j$ denotes the canonical injection of $\mathcal{P}_{0,c}$ in $\mathcal{M}(C^*(G))$ and $q: \Psi^*(G)\rightarrow\Psi^*(G)/C^*(G)$ denotes the quotient map.
\end{proposition}
\begin{proof}
    We know that the map $\sigma : \mathcal{P}_{0,c}\to \Sigma_c^0(N^*) \simeq C_c^\infty(S^*N)$ is surjective and has kernel $\mathcal{P}_{-1,c}$, so if $\sigma(P_1) = \sigma (P_2),$ then $j(P_1-P_2) \in C^*(G)$ so $q\circ j (P_1)= q \circ j(P_2)$. We conclude by using the previous proposition and the density of $C_c^\infty(S^*N)$ in $C_0(S^*N)$.
\end{proof}

\subsection{Extensions to multipliers}
The following result is classical  (see \emph{e.g.} \cite{akemann1973multipliers, pedersen1979c}). 

\begin{proposition}\label{prop2.3}
Suppose $f:A\to B$ is an onto morphism of $C^*$-algebras. Then $f$ extends uniquely to a $*$-homomorphism $\Tilde{f}:\mathcal{M}(A)\to \mathcal{M}(B)$, and we get a commutative diagram 
\begin{center}
  \begin{tikzcd}
    B\arrow{r} & \mathcal{M}(B)\arrow{r}{q_B} & \mathcal{M}(B)/B \\
    A\arrow{r}\arrow{u}{f}  & \mathcal{M}(A)\arrow{r}{q_A}\arrow{u}{\Tilde f}  &\mathcal{M}(A)/A \arrow{u}{\overline{f}} \\
\end{tikzcd}  
\end{center}
\end{proposition}

Denote by $j_r:\mathcal{P}_{0,c}\to \mathcal{M}(C^*_r(G))$ the composition of $j$ with the extension to the multipliers of the canonical morphism $\lambda:C^*(G)\to C^*_r(G)$. It is the unique extension to $\mathcal{P}_{0,c}$ of the inclusion $C_c^\infty(G)\subset C^*_r(G)$.

If $P\in \mathcal{P}_{0,c}$ then $\Tilde{\lambda}j(P)=j_r(P)$. The following diagram commutes:

\begin{center}
\begin{tikzcd}[row sep=huge, column sep=huge]
    \mathcal{P}_{0,c} \arrow[rd, "j_r"'] \arrow[r, "j"] & \mathcal{M}(C^*(G)) \arrow[d, "\Tilde \lambda"] \\
    & \mathcal{M}(C^*_r(G))
\end{tikzcd}
\end{center}

By definition, $\Psi^*(G)$ (\emph{resp.} $\Psi^*_r(G)$) is the closure of $j(\mathcal{P}_{0,c})$ in $\mathcal{M}(C^*(G))$ (\emph{resp.} $\mathcal{M}(C^*_r(G))$).  Hence, the above diagrams imply the existence of a morphism
\begin{equation*}
    \overline{\lambda}:\Psi^*(G)/C^*(G)\longrightarrow \Psi^*_r(G)/C^*_r(G),
 \end{equation*}
such that $\overline{\lambda}\circ q\circ j=q_r\circ j_r$.

\subsection{Approximation of the principal symbol}

The last step in the proof of our main result is provided by the following Proposition.

\begin{proposition}\label{02}
Let $G\rightrightarrows M$ be a Lie groupoid. Let $x \in M$ and $\xi \in T^*_{x} M \backslash \{0\}$. Denote by $\ell$ the half line $\R_+^*\xi$. There exists a sequence of unit vectors  $\phi_n\in L^2(G_{x})$ such that, for every $P\in \mathcal{P}_{0,c}$, we have $$\lim_{n\to \infty} \langle \phi_n|\lambda_x (P)\phi_n\rangle=\sigma(P)(x,\ell).$$
 \end{proposition}
 
 We will choose $\phi_n$ concentrated near $x$ and whose differential concentrates near $\ell$. 
 
 A first requirement is that $\lim_{n\to \infty} \langle \phi_n|\lambda_x (f)\phi_n\rangle=0$ if $f\in C_c^\infty G)$. This will come from the following Lemma which shows this remains true when $f\in C_r^*G)$.
 
 \begin{lemma}\label{220}
Let $(f_n)_{n\in \mathbb N}$ be a sequence of unit vectors in $L^2(G_x)$ with compact support $(C_n)_{n\in \mathbb N}$ whose measure tends to $0$. If $P\in C^*_r(G)$  then $\lim_{n\to \infty} \langle f_n|\lambda_x (P)f_n\rangle=0$, where $\lambda_x:C^*_r(G)\to B(L^2(G_x))$.
\end{lemma}
\begin{proof}
Let $K\in C^\infty_c(G)$. Suppose $(\chi_n)_{n\in \mathbb N}$ is a sequence of characteristic functions of $(C_n)_{n\in \mathbb N}$, then,
\begin{equation*}
    \langle f_n\,|\,\lambda_x(K)f_n\rangle = \int_{G_x\times G_x}\overline{f_n(\gamma_1)}K(\gamma_1 \gamma_2^{-1})f_n(\gamma_2)\chi_n(\gamma_1)\chi_n(\gamma_2)\mathrm{d}\gamma_1 \mathrm{d}\gamma_2 .
\end{equation*}
Put $K^\prime(\gamma_1,\gamma_2)=K(\gamma_1\gamma_2^{-1})\chi_0(\gamma_1)\chi_0(\gamma_2)$  we have
\begin{equation*}\label{K_n}
    |\langle f_n\,|\,\lambda_x(K)f_n\rangle| \leq \|f_n\otimes\overline{f_n}\|_2\cdot\|(\chi_n\otimes\chi_n)~ K^\prime\|_2 = \|(\chi_n\otimes\chi_n)~ K^\prime\|_2.
\end{equation*}
$(\chi_n)_{n\in \mathbb N}$ tends to $0$ \emph{a.e.} therefore the last norm tends to $0$. 

As $|\langle f_n\,|\,\lambda_x(K)f_n\rangle|\le \|K\|$ the set $\{K\in C_r^*(G);\ \langle f_n\,|\,\lambda_x(K)f_n\rangle\to 0\}$ is closed so that $\langle f_n\,|\,\lambda_x(K)f_n\rangle\to 0$ for every $K\in C_r^*(G)$.
\end{proof}

Thanks to Lemma \ref{220}, we will be able to assume that $P$ is given by a local formula as in definition \ref{340}. We are therefore given an open subset $U\subset G$ containing $M$, an open subset $U'\subset N$ containing $\{(y,0);\ y\in M\}$ and a diffeomophism $\theta : U' \to U$  such that
\begin{itemize}
    \item $\theta(y,0)=y$ for all $y\in M$ and $r(\theta(y,X))=y$ for $(y,X)\in U'$;
    \item $p_y(\mathrm d\theta_{(y,0)}(X))=X $ for all $y\in M$ and $X\in N_y$.
\end{itemize}
and the action of $P$ is given by an expression
\begin{equation*}
  \lambda_x(P)(\phi)(\gamma)=
\frac{1}{(2\pi)^n}\int_{N^*_{y}}\left(\int_{G_{x}}
    \mathrm e^{-\mathrm i\langle \theta^{-1}(\gamma \eta^{-1}),\xi\rangle}\,
    a(y,\xi )\,\chi(\gamma \eta^{-1})\,\phi(\eta)\,
    \mathrm d\eta\right)\,\mathrm d\xi,
    \end{equation*}
where  $\chi$ is a cut-off function, $y=r(\gamma)$, $n=\dim N_y$ and $a\sim \sum_{k=0}^{+\infty}a_{-k}$ is a \emph{classical symbol} of order $0$ whose first term is $a_0= \sigma(P).$

Let $W$ be a small enough neighborhood of $x$ in $G_x$ such that $W\subset U$, $\gamma_1\gamma^{-1}_2\in U$ for every $(\gamma_1,\gamma_2)\in W\times W$. Note that $\theta^{-1}(\gamma_1\gamma^{-1}_2)\in N_{r(\gamma)}\simeq T_\gamma G_x=T_\gamma W$. We further assume that the map  $(\gamma_1,\gamma_2)\mapsto \big(\gamma_1,\theta^{-1}(\gamma_1\gamma^{-1}_2)\big)$ is a diffeomorphism from $W\times W$ onto a neighborhood of $W$ in $TW$. It follows that for $\chi'\in C_c^\infty(W\times W)$ the operator $Q$ given by the formula
\begin{equation*}
Q(\phi)(\gamma)
  =\frac{1}{(2\pi)^n}\int_{N^*_{y}}\left(\int_{G_x}
    \mathrm e^{-\mathrm i\langle \theta^{-1}(\gamma \eta^{-1}),\xi\rangle}\,
    a(y,\xi)\,\chi(\gamma \eta^{-1})\chi'(\gamma,\eta)\,\phi(\eta)\,
    \mathrm d\eta\right)\,\mathrm d\xi,
\end{equation*}
($y=r(\gamma)$), is a pseudodifferential operator on $W$.  

Take then $\phi$ with compact support in $W$ and $\chi'=1$ on $\mathrm{supp}(\phi)\times \mathrm{supp}(\phi)$. We find $\langle \phi,\lambda_x(P)(\phi)\rangle =\langle \phi_n,Q(\phi_n)\rangle$. Note also that $\sigma(P)(x,\ell)=\sigma(Q)(x,\ell)$.

The map $\gamma \mapsto \theta^{-1}(\gamma^{-1})$ identifies $W$ with an open subset of the euclidean space $N_x$.

Therefore, the following Lemma ends the proof of Prop. \ref{02} 

\begin{lemma}
Let $B$ be an open ball in $\R^k$ and let $\phi$ be a smooth nonnegative function with compact support in $B$ with $L^2$-norm equal to $1$. Let $\xi_0\in \R^k\setminus \{0\}$ 
and put $\phi_n (y)= n^{\frac{k}{2}} \phi(ny) \mathrm e^{i n^2\langle y,\xi_0\rangle }.$\\
 For every pseudodifferential operator $Q$ with compact support in $B\times B$, we have $$\lim_{n\to \infty} \langle \phi_n,Q(\phi_n)\rangle =\sigma(Q)(0,\xi_0).$$
\end{lemma}

 \begin{proof}
 We may write 
 \begin{equation*}
Q(\phi)(x)
  =\frac{1}{(2\pi)^k}\int_{\R^k}\left(\int_{B}
    \mathrm e^{\mathrm i\langle x-y,\xi\rangle}\,
   b(x,\xi)\,\chi(x)\chi(y)\,\phi(y)\,
    \mathrm dy\right)\,\mathrm d\xi,
\end{equation*}
where $b$ is a classical symbol of order $0$ and $\chi $ is a smooth function with compact support in $B$ and equal to $1$  on the supports of all the $\phi_n$.

\medskip Write $b(x,\xi)=b(0,\xi)+\sum_{j=1}^kx_jb_j(x,\xi)$ and set \begin{equation*}
Q_0(\phi)(x)
  =\frac{1}{(2\pi)^k}\int_{\R^k}\left(\int_{B}
    \mathrm e^{\mathrm i\langle x-y,\xi\rangle}\,
   b(0,\xi)\,\chi(x)\chi(y)\,\phi(y)\,
    \mathrm dy\right)\,\mathrm d\xi.
\end{equation*}
Now $\|x_j\phi_n\|_2\to 0$, and therefore $\langle \phi_n,(Q-Q_0)(\phi_n)\rangle\to 0$.

Also, as $\chi \phi_n=\phi_n$, we find 
\begin{eqnarray*}\langle \phi_n,Q_0(\phi_n)\rangle&=&\frac{1}{(2\pi)^k}\int_{B}\overline{\phi_n(x)}\left(\int_{\R^k}\left(\int_{B}
    \mathrm e^{\mathrm i\langle x-y,\xi\rangle}\,
   b(0,\xi)\,\,\phi_n(y)\,
    \mathrm dy\right)\,\mathrm d\xi\right)\mathrm d x\\
    &=&\frac{1}{(2\pi)^k}\int_{B\times \R^k}\overline{\phi_n(x)}\mathrm e^{\mathrm i\langle x,\xi\rangle} b(0,\xi)\widehat{\phi_n}(\xi)\,dx\,d\xi\\
&=&\frac{1}{(2\pi)^k}\int_{\R^k}\overline{\widehat{\phi_n}(\xi)}b(0,\xi)\widehat{\phi_n}(\xi)\,d\xi\end{eqnarray*}
Using the change of variables $y\to ny$, we find \begin{eqnarray*}\widehat{\phi_n}(\xi)&=&n^{k/2}\int_{B}
    \mathrm e^{-\mathrm i\langle y,\xi\rangle}\,
   \,\phi (ny)e^{i n^2\langle y,\xi_0\rangle}\,dy\\
   &=&n^{-k/2}\hat\phi(\xi/n-n\xi_0),
\end{eqnarray*}
therefore \begin{eqnarray*}\langle \phi_n,Q_0(\phi_n)\rangle&=&\frac{1}{(2\pi)^k}\int_{\R^k}n^{-k}|\hat\phi(\xi/n-n\xi_0)|^2b(0,\xi)\,d\xi\\
&=& \frac{1}{(2\pi)^k}\int_{\R^k}|\hat\phi^2(\xi)|b(0,n\xi+n^2\xi_0)\,d\xi
\end{eqnarray*}
Now, $\frac{1}{(2\pi)^k}\int_{\R^k}|\hat\phi^2(\xi)|\,d\xi=1$ and $b$ is bounded. As $b(0,n\xi+n^2\xi_0)\to b_0(0,\xi_0)$, where $b_0$ is the principal part of the symbol, the Lemma follows from Lebesgue dominated convergence Theorem.
\end{proof}

\subsection{The exact sequences}
It follows from Prop. \ref{02} that, for $P\in \mathcal{P}_{0,c}$ and $(x,\xi)\in S^*N$, we have $|\sigma(x,\xi)|\le \|q_r(j_r(P))\|$, and therefore $\|\sigma(P)\|_\infty\le \|q_r(j_r(P))\|$.

So the map $\sigma $ extends by continuity and we obtain a morphism
\begin{equation*}
    \beta:\Psi^*_r(G)/C^*_r(G)\to C_0(S^*N)
\end{equation*}
such that $\sigma(P)=\beta\circ q_r\circ j_r(P)$.

Using Lemmas \ref{02} and \ref{03} we have
\begin{equation*}
    \|P\|_{\mathcal{M}(C^*(G))/C^*(G)}\leq  \|P\|_{\mathcal{M}(C_r^*(G))/C_r^*(G)}.
\end{equation*}
\begin{theorem}
The map $\phi: \Psi^*(G)/C^*(G)\to\Psi_r^*(G)/C_r^*(G) $ is a bijection, in addition, we have the following short exact sequences of $C^*$-algebras:
\begin{center}

\begin{tikzcd}
0 \arrow[r] & C^*(G) \arrow[r]\arrow[d,"\lambda"]   & \Psi^*(G) \arrow[r,"\sigma"]\arrow[d,"\overline \lambda"]    & C_0(S^*N) \arrow[r] \arrow[d, equal] & 0 \\
0 \arrow[r] & C^*_r(G) \arrow[r] & \Psi^*_r(G) \arrow[r,"\sigma"] & C_0(S^*N) \arrow[r]                                & 0
\end{tikzcd}
\end{center}
\end{theorem}
\begin{proof}
  We have constructed morphisms:
\begin{align*}
    \alpha:C_0(S^*N)\longrightarrow& ~\Psi^*(G)/C^*(G),\\
    \overline{\lambda}: \Psi^*(G)/C^*(G)\longrightarrow& ~\Psi^*_r(G)/C^*_r(G),\\
    \beta: \Psi^*_r(G)/C^*_r(G)\longrightarrow& ~C_0(S^*N),
\end{align*}
such that the following diagram commutes: 
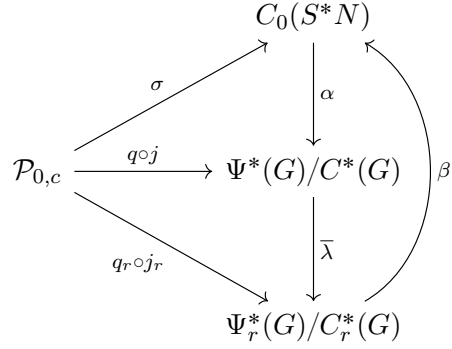
\begin{center}
    \begin{tikzcd}
                  &  & C_0(S^*N) \arrow[dd, "\alpha"]                          \\
 &  &      \\
{\mathcal P_{0,c}} \arrow[rr, "q\circ j"] \arrow[rruu, "\sigma"] \arrow[rrdd, "q_r\circ j_r"'] &  & \Psi^*(G)/C^*(G) \arrow[dd, "\overline\lambda"]                         \\
&  &           \\
  &  & \Psi^*_r(G)/C^*_r(G) \arrow[uuuu, "\beta"', bend right=60, shift right]
\end{tikzcd}
\end{center}
Since the maps $\sigma$, $q\circ j$, and $q_r\circ j_r$ have dense range it follows that
\begin{align*}
    \alpha\circ \beta\circ\overline{\lambda}&= id_{\Psi^*(G)/C^*(G)},\\
    \beta \circ \overline{\lambda}\circ\alpha &= id _{C_0(S^*N)},\\
    \overline{\lambda}\circ\alpha\circ\beta&=id_{\Psi^*_r(G)/C^*(G)}.
\end{align*}
In particular $\alpha$, $\beta$, and $\overline{\lambda}$ are isomorphisms.
\end{proof}

%***************************************************************************

\begin{acknowledgements}
The author would like to sincerely thank her PhD advisors, Georges Skandalis and
Stéphane Vassout, for their support, guidance, and numerous valuable
insights throughout this work.
\end{acknowledgements}
\newpage

%-----------------------------------------------------------------------------
% Beginning of biblio.tex
%-----------------------------------------------------------------------------

%*************************************

%-----------------------------------------------------------------------------
% End of biblio.tex
%-----------------------------------------------------------------------------
\end{document}